\documentclass[11pt]{amsart}

\usepackage{latexsym}
\usepackage{amsmath}
\usepackage{amssymb}

\title[Primary decompositions of Frobenius powers of ideals]{Growth of
primary decompositions of \\
Frobenius powers of ideals}
\author{Trung T. Dinh}

\newtheorem{theorem}{Theorem}[section]
\newtheorem{lemma}[theorem]{Lemma}
\newtheorem{proposition}[theorem]{Proposition}
\newtheorem{corollary}[theorem]{Corollary}
\theoremstyle{definition}

\newcommand{\ass}{\text{Ass }}
\newcommand{\mbf}{\mathbf}
\newcommand{\mf}{\mathfrak}
\newcommand{\ds}{\displaystyle}
\newcommand{\Min}{\text{Min}}

\begin{document}
\maketitle


\begin{abstract} It was previously known, by work of Smith-Swanson
and of Sharp-Nossem, that the linear growth property of primary
decompositions of Frobenius powers of ideals in rings of prime
characteristic has strong connections to the localization problem in
tight closure theory. The localization problem has recently been settled
in the negative by Brenner and Monsky, but the linear growth question
is still open. We study growth of primary decompositions of
Frobenius powers of dimension one homogeneous ideals in graded rings
over fields. If the ring is positively graded we prove that the linear
growth property holds. For non-negatively graded rings we are able to
show that there is a ``polynomial growth''. We present explicit primary
decompositions of Frobenius powers of an ideal, which were known to have
infinitely many associated primes, having this linear growth property.
We also discuss some other interesting examples.
\end{abstract}


\section{Introduction}


Let $R$ be a Noetherian commutative ring of prime characteristic
$p$. Let $I$ be an ideal of $R$. For each $q=p^e$ the $q$th
\emph{Frobenius power} of $I$, denoted by $I^{[q]}$, is the ideal of $R$
generated by $q$th powers of generators of $I$. We say that the
Frobenius powers of $I$ have \emph{linear growth of primary decompositions}
if there exists a number $c$ such that, for each $q=p^e$, there is
a primary decomposition $$I^{[q]}=Q_1\cap Q_2\cap\cdots \cap Q_r$$
with the property that $(\sqrt{Q_i})^{c[q]}\subseteq Q_i$ for all
$i=1,\ldots, r$. (This number $r$ depends on $q$).

The linear growth property of primary decompositions of Frobenius powers
of ideals has strong connections to the localization problem in tight
closure theory. This connection was first discovered by Smith and
Swanson in \cite{swanson-smith}. In that paper they proved that if
the linear growth property holds for an ideal, then tight closure
commutes with the localization of that ideal at one element. Later,
Sharp and Nossem ~\cite{sharp-nossem} showed that, under the additional
 hypothesis that the ring $R$ has a test element, and the set
$\bigcup_{q=p^e}\ass R/I^{[q]}$ is finite, tight closure of $I$
 commutes with localization at arbitrary multiplicative subsets as long as
the linear growth property holds for the Frobenius powers of $I$.

Beside its importance in tight closure theory, the problem of investigating
the linear growth property of primary decompositions of Frobenius powers
is interesting in its own right. One of the reasons is that the ordinary powers
of ideals were shown to have the linear growth property. (The definition of
linear growth property for ordinary  powers is obtained by replacing
Frobenius powers in the statement by the corresponding ordinary powers).
This fact was first proved by Swanson ~\cite{swanson}, and Sharp ~\cite{sharp}
and Yao ~\cite{yao} later independently gave shorter proofs. Though these
proofs use different methods, they rely on the technique  of passing to
Rees algebras to reduce the problem to the case of principal ideals generated
by non-zero divisors, and the fact that for any ideal $I$ in a Noetherian
ring $R$, the set $\bigcup_{n>0}\ass R/I^n$ is finite. One cannot hope
to use similar ideas to attack the linear growth question of Frobenius
powers of ideals, for there is no analogue of the concept of Rees algebras
associated to Frobenius powers of ideals. Moreover, the set
$\bigcup_{q=p^e} \ass R/I^{[q]}$ may be infinite due to examples of
Katzman ~\cite{katzman} and of Singh and Swanson ~\cite{swanson-singh}.

It is an open question whether every ideal in Noetherian rings of prime characteristic has the
linear growth property. There are no known examples of ideals for which the linear growth property
does not hold. Due to the lack of techniques available in the literature to attack the problem, there has been a very limited effort
to show that certain (classes of) ideals have the linear growth property.  To the author's
knowledge the only work toward this question, besides the very general
statement of Sharp and Nossem mentioned above, was the work
of Smith and Swanson in ~\cite{swanson-smith}. They
proved that monomial ideals in a monomial ring have the linear
growth property. Part of that paper was also devoted to analyzing primary
decompositions of the Frobenius powers of the ideal $I=(x,y)$
in the hypersurface
$$
\ds R=\frac{k[t,x,y]}{\big(xy(x-y)(x-ty)\big)},
$$
in which $k$ is a field of positive characteristic $p$. This is the
well-known example with the property that the set $\bigcup_{q=p^e} \ass R/I^{[q]}$
is infinite, due to Katzman ~\cite{katzman}. Smith and Swanson
proved that this ideal have linear growth of primary decompositions.
In \cite[\S 3]{swansontenlectures}, Swanson raised the question whether for each concrete ideal
$I$ available in the literature with the property that $\bigcup_{q=p^e}\ass R/I^{[q]}$ is infinite, we can
find ``good''primary decompositions, namely decompositions growing
linearly.

This paper was inspired by the work of Smith-Swanson and was partially
motivated by the above question of Swanson. It is thus mainly concerned
with homogeneous ideals of dimension one in graded rings over fields of prime
characteristics. In Section 2 we give a short proof of the
fact that for dimension one homogeneous ideals in a finitely
generated, positively graded ring over a field of positive characteristic, the linear growth
property holds.

Section 3 studies dimension one homogenous ideals in non-negatively graded, finitely generated rings
over fields of positive characteristics. We are not able to show that the Frobenius powers
of these ideals have linear growth of primary decompositions, but are able to prove that there are primary decompositions
 having polynomial growth. A consequence of this result, however, produces a class of ideals
having the linear growth property, which includes Katzman's ideal as a special case.

Section 4 analyzes primary decompositions of the Frobenius powers of
the ideal \linebreak $I=(u,v,x,y)$ in the hypersurface
$$
R=\ds\frac{k[t,u,v,x,y]}{(u^2x^2+tuvxy+v^2y^2)}.
$$
This hypersurface was studied by Singh and Swanson in
~\cite{swanson-singh}. By a slight modification of their proofs, we show
 that the set $\bigcup_{q=p^e}\ass R/I^{[q]}$ is infinite. We prove
 that this ideal has the linear growth property.

Sections 5 and 6 are reserved for discussions of some other
examples. In particular, we would like to pose a question on the
finiteness of the set of associated primes of the Frobenius powers
of the ideal studied by Monsky in \cite{monsky}, which was used recently by Brenner and
Monsky in ~\cite{brenner-monsky} to disprove the localization conjecture for tight closure.
This is an interesting question since its answer will likely give us some information on the
validity of the linear growth question.


\section{Dimension one homogeneous ideals in positively graded rings}

It was shown independently by Huneke ~\cite{huneke} and Vraciu
~\cite{vraciu} that, for dimension one homogeneous ideals in a
finitely generated, positively graded ring over a field of positive
characteristic, tight closure commutes with localization. Using some of the main
results in these papers, we show in this section that such ideals have the linear
growth property. In view of the result of Sharp and Nossem mentioned in the
introduction, this gives an alternative proof of the commutativity of tight closure
and localization for this class of ideals.

We first, however, start with a remark that for the linear growth question
in general, we need only find bounds for the embedded components.


\begin{proposition}\label{embedded} Let $R$ be a Noetherian ring of
prime characteristic $p$ and $I$ an ideal of $R$. Then there is a
number $c$ such that for each $q=p^e$ and for each primary
decomposition
$$I^{[q]}=Q_1\cap Q_2\cap\cdots\cap Q_r,$$
we have $(\sqrt{Q_i})^{c[q]}\subseteq Q_i$ for all isolated components
$Q_i$.
\end{proposition}


\begin{proof} If $P$ is a minimal prime over $I$, then $IR_P\cap R$
is the unique isolated component of $I$ whose radical is $P$, hence
$P=\sqrt{IR_P\cap R}$. There are only a finite number of minimal
primes over $I$, thus we can find a number $c$ such that
$P^c\subseteq IR_P\cap R$ for all $P$ minimal over $I$. Now let
$q=p^e$. For each prime ideal $P$ minimal over $I$ we have
$Q=I^{[q]}R_P\cap R$ is the unique isolated component of $I^{[q]}$
whose radical is $P$, and
$$(\sqrt{Q})^{c[q]}=P^{c[q]}\subseteq (IR_P\cap R)^{[q]}\subseteq
I^{[q]}R_P\cap R = Q.$$
This gives us the conclusion of the proposition.
\end{proof}


\begin{theorem} Let $R$ be a finitely generated, positively graded
ring over a field of prime characteristic $p$. Let $I$ be a
homogeneous ideal of height $d-1$, where $d$ is the dimension of $R$.
 Then the Frobenius powers of $I$ have linear growth of primary decompositions.
\end{theorem}


\begin{proof} Denote by $\mf{m}$ the maximal homogeneous ideal of
$R$. Let $z\in \mf{m}$ be such that it is not contained in any minimal prime
of $I$. We claim that there is a constant $c$ such that
$$
I^{[q]}:z^\infty=I^{[q]}:z^{cq}
$$
for all $q=p^e$. Here we define $I^{[q]}:J^\infty=\bigcup_{n>0}I^{[q]}:J^n$ for each ideal $J$ of $R$.

To prove the claim, we note that it was proved in ~\cite{vraciu}, or
~\cite{huneke}, that there is a constant $c$ such that
$$
I^{[q]}:\mf{m}^\infty=I^{[q]}:\mf{m}^{cq}
$$
for all $q=p^e$. Let $a\in I^{[q]}:z^\infty$ be an arbitrary element. Then
$az^n\in I^{[q]}$ for some $n$. Since $z$ is not in any minimal
prime of $I$, the ideal $(z^n,I^{[q]})$ contains some power of
$\mf{m}$. Hence
$$
a\in I^{[q]}:\mf{m}^\infty = I^{[q]}:\mf{m}^{cq}\subseteq I^{[q]}:z^{cq}.
$$
Now the claim is proved, and it follows that
$$
I^{[q]}=(I^{[q]}+Rz^{cq})\cap (I^{[q]}:z^\infty).
$$
Since $z$ is not contained in any minimal prime of $I$, the ideal
$I^{[q]}:z^\infty$ is precisely the intersection of all isolated
components of any primary decomposition of $I^{[q]}$, and the ideal
$I^{[q]}+Rz^{cq}$ is primary to the homogeneous maximal ideal
$\mf{m}$. Pick a number $c'$ such that $I+Rz^c\supseteq
\mf{m}^{c'}$. Then
$$
I^{[q]}+Rz^{cq}\supseteq \mf{m}^{c'[q]}.
$$
Therefore, combining with Proposition ~\ref{embedded}, from the
decomposition
$$
\ds I^{[q]}=(I^{[q]}+Rz^{cq})\cap (I^{[q]}:z^\infty),
$$ we obtain primary decompositions with the desired property.
\end{proof}

\section{Dimension one homogeneous ideals in non-negatively graded rings}


Let $k$ be a field of positive characteristic $p$. We define an
$\mathbb{N}$-grading on the polynomial ring $k[t,x_1,\ldots,x_n]$ as follows:
 $\deg t=0$ and $\deg x_i=1$ for all $i=1,\ldots, n.$ In this section
 we study primary decompositions of the Frobenius powers of the
dimension one ideal $I=(x_1,\ldots, x_n)$ in the graded ring
$$
R=\frac{k[t,x_1,\ldots,x_n]}{(f_1,f_2,\ldots,f_r)}
$$
in which $f_1,f_2,\ldots, f_r$ are homogeneous polynomials. We are
not able to prove that the linear growth property holds in this case, but
 are able to establish polynomial growth. We will need some preparation.


The following lemma was extracted from the proof of Lemma 2.1 in ~\cite{swanson-smith}.

\begin{lemma}\label{stable}
Let $k$ be a field of prime characteristic $p$ and
$$
R=\frac{k[t,x_1,\ldots,x_n]}{(f_1,\ldots,f_r)},
$$
where $t,x_1,\ldots,x_n$ are variables and $f_1,\ldots,f_r$ are
homogeneous polynomials (with respect to the grading defined earlier)
 of positive degrees. Let $I=(x_1,\ldots,x_n)$. Suppose that for
 each $q=p^e$ there is a polynomial $h_q(t)\in k[t]$ such that the
ideal $I^{[q]}:h_q(t)$ is primary to the ideal $(x_1,\ldots,x_n)$.

Let $h_q(t)=\tau_{1}^{s_{1}}\tau_2^{s_2}\cdots\tau_l^{s_l}$ be the decomposition of $h_q(t)$
into a product of powers of distinct irreducible polynomials. (This number $l$ may depend on $q$). Then
there is a constant $c$ such that for each $q=p^e$ the ideal
$I^{[q]}$ has a primary decomposition
$$
I^{[q]}=Q\cap Q_1\cap\cdots\cap Q_l,
$$
in which $Q$ is primary to $(x_1,\ldots,x_n)$, and $Q_i=I^{[q]}+R\tau_i^{s_i}$ is
primary to $(x_1,\ldots,x_n,\tau_i)$ for each $i=1,\ldots,l$. Furthermore,
$$
(\sqrt{Q_i})^{cq+s_i}\subseteq Q_i
$$
for all $i=1,\ldots,l$.
\end{lemma}


\begin{proof}
Saying that $I^{[q]}:h_q(t)$ is primary to $(x_1,\ldots,x_n)$ is equivalent to
saying that
$$
I^{[q]}:h_q(t)=\big(I^{[q]}:h_q(t)\big):g(t)
$$
for all $g(t)\in k[t]\setminus \{0\}$. In particular when applied to
$g(t)=h_q(t)$ we obtain that
$$
I^{[q]}=\big(I^{[q]}+Rh_q(t)\big)\cap \big(I^{[q]}:h_q(t)\big).
$$
The ideal $Q=I^{[q]}:h_q(t)$ is primary to
$(x_1,\ldots,x_n)$ as defined. We want to analyze primary decompositions of the
ideal $I^{[q]}+Rh_q(t)$.

The ideal $I^{[q]}+Rh_q(t)$ has dimension 0, thus
every minimal prime must be maximal. Each maximal ideal containing
$I^{[q]}+Rh_q(t)$ must be of the form $(x_1,\ldots, x_n,\tau_{i})$ for
some $i$. We have the following obvious inclusion
$$
I^{[q]}+Rh_q(t)\subseteq (I^{[q]}+R\tau_{1}^{s_{1}})\cap (I^{[q]}+R\tau_{2}^{s_{2}})
\cap\cdots\cap (I^{[q]}+R\tau_{l}^{s_{l}}).
$$
\noindent For the reverse inclusion, let $Q_i$ be the unique
$(x_1,\ldots,x_n,\tau_{i})$-primary component of \linebreak $I^{[q]}+Rh_q(t)$. Then
$$
Q_i = \big(I^{[q]}+Rh_q(t)\big)R_{(x_1,\ldots,x_n,\tau_{i})}\cap
R\supseteq (I^{[q]}+R\tau_{i}^{s_{i}}).
$$
Hence
$$
I^{[q]}+Rh_q(t)=(I^{[q]}+R\tau_{1}^{s_{1}})\cap (I^{[q]}+R\tau_{2}^{s_{2}})
\cap\cdots\cap (I^{[q]}+R\tau_{l}^{s_{l}}).
$$

Our discussion so far leads to the following primary decomposition
\begin{eqnarray*}
I^{[q]} &=& \big(I^{[q]}:h_q(t)\big)\cap(I^{[q]}+R\tau_1^{s_1})\cap
(I^{[q]}+R\tau_2^{s_2})\cap\cdots\cap (I^{[q]}+R\tau_l^{s_l})\\
&=&Q\cap Q_1\cap Q_2\cap\cdots\cap Q_l,
\end{eqnarray*}
\newpage
\noindent in which $\sqrt{Q_i}=\big(x_1,\ldots,x_n,\tau_{i}\big)$, and
clearly
$$
(\sqrt{Q_i})^{nq+s_i}=\big(x_1,\ldots,x_n,\tau_i\big)^{nq+s_i}\subset
(x_1^q,\ldots,x_n^q,\tau_i^{s_i}) = I^{[q]}+R\tau_i^{s_i},
$$
as desired.
\end{proof}


The lemma above in some sense says that there are primary decompositions
of the Frobenius powers of the ideal $I$ which grow at most as fast as
the ``primary decompositions'' of the polynomials $h_q(t)$. The proposition
 below shows that there are such polynomials whose ``primary decompositions''
grow polynomially. We first need a lemma.


\begin{lemma}\label{minor}
Let $R$ be a domain and $(a_{ij}:i=1,\ldots,k, j=1,\ldots,l)$ a matrix with
entries in $R$. Let $b_1,\ldots,b_l$ be elements in the quotient field of
$R$ such that for each $i=1,\ldots, k$, the sum $a_{i1}b_1+\ldots+a_{il}b_l$
is an element of $R$ and is divisible by every non-zero minor of the matrix
$(a_{ij})$. Then there exist $b'_1,b'_2,\ldots,b'_l$ in $R$ such that
$a_{i1}b_1+\cdots+a_{il}b_l=a_{i1}b'_1+\cdots+a_{il}b'_l$ for all $i=1,\ldots, k$.
\end{lemma}


\begin{proof}
We may assume that $k\leq l$ and that the matrix $(a_{ij})$ has maximal rank. We may then
assume that the submatrix $(a_{ij}: 1\leq i,j\leq k)$ has nonzero determinant. Now
let $(b'_1,\ldots,b'_l)$ be the unique solution in the quotient
 field of $R$ of the linear system
$$\begin{bmatrix}
a_{11}&a_{12}&\ldots&\ldots&a_{1l}\\
a_{21}&a_{22}&\ldots&\ldots&a_{2l}\\
.&.&\ldots&\ldots&.\\
a_{k1}&a_{k2}&\ldots&\ldots&a_{kl}\\
0&\ldots&1&\ldots&0\\
.&.&\ldots&\ldots&.\\
0&\ldots&0&\ldots&1
\end{bmatrix}
\begin{bmatrix}
x_1\\
x_2\\
.\\
.\\
.\\
.\\
x_l
\end{bmatrix}
=
\begin{bmatrix}
a_{11}b_1+\cdots+a_{1l}b_l\\
a_{21}b_2+\cdots+a_{2l}b_l\\
.\\
a_{k1}b_l+\cdots+a_{kl}b_l\\
0\\
.\\
0
\end{bmatrix}.
$$

Using Cramer's rule and the fact that the determinant of the coefficient matrix
divides every entry of the column vector on the right-hand side, we can conclude
that this solution is in $R$.
\end{proof}

\begin{proposition}\label{existence}
Let $S=k[t,x_1,\ldots,x_n]$ be a polynomial ring in $n+1$ variables
over a field $k$. Define the grading as in the beginning of the section. Let $f_1,\ldots,f_r$ be
homogeneous polynomials of positive degrees. Then for each positive integer $q$ there is a polynomial
\linebreak $h_q(t)$ in $k[t]$ such that the ideal $(x_1^q,x_2^q,\ldots,x_n^q,f_1,\ldots,f_r):_Sh_q(t)$
is primary to $(x_1,\ldots,x_n)$.

Moreover these polynomials $h_q$ can be chosen to have the following
property: there is a constant $c$ not depending on $q$ such that for
each decomposition $h_q(t)=\tau_1^{s_1}\tau_2^{s_2}\cdots \tau_l^{s_l}$
into a product of powers of distinct irreducible polynomials, the numbers
$s_j$ are bounded above by $cq^{n-1}$.
\end{proposition}


\begin{proof}
To simplify the notation let $\mbf{x}$ denote $x_1,\ldots,x_n$, and
for each $\mbf{u}=(u_1,\ldots,u_n)\in \mathbb{N}^n$, where $\mathbb{N}$ denotes the
set of non-negative integers, let $\mbf{x}^\mbf{u}$ stand for $x_1^{u_1}\cdots x_n^{u_n}$,
$|\mbf{u}|$ stand for the integer $u_1+u_2+\cdots+u_n$ and $||\mbf{u}||$ denote $\max\{u_i: i=1,\ldots,n\}$.

Let $\deg f_i=d_i$ for each $i=1,\ldots,r$ and write
$$
\ds f_i=\sum_{|\mbf{v}|=d_i}A_{i,\mbf{v}}\mbf{x}^\mbf{v},
$$
where $A_{i,\mbf{v}}\in k[t]$ for all $i$ and $\mbf{v}$. We call the $\mbf{x}^\mbf{v}$ the monomial terms of $f_i$.

For each positive integer $d$ we define the matrix $M_d$ as follows. The rows of
this matrix are indexed by the vectors $\mbf{u}\in \mathbb{N}^n$ with $|\mbf{u}|=d$
 and $||\mbf{u}||<q$. Its columns are indexed by the vectors
$\mbf{w}^i\in \mathbb{N}^n$ with $|\mbf{w}^i|=d-d_i$ and $i=1,\ldots,r$. For each such
$\mbf{u}$ and such $\mbf{w}^i$ the $(\mbf{u},\mbf{w}^i)$-entry is $A_{i,\mbf{v}}$
if $\mbf{w}^i+ \mbf{v}=\mbf{u}$ and is 0 otherwise. Now we define $h_q(t)$ to be
the least common multiple of all the non-zero minors of all matrices $M_d$
with $1\leq d\leq n(q-1)$.

It is clear that there is a number $\alpha$ such that the sizes of the matrices
$M_d$ with \linebreak $d=1,\ldots, n(q-1),$ are bounded above by $\alpha q^{n-1}$, for all $q$.
 Hence the nonzero minors of these matrices are polynomials in $t$ of degrees bounded above by $\beta q^{n-1}$
for some $\beta$ not depending on $q$. The polynomial $h_q(t)$ was defined to be the least common multiple
of all these polynomials, thus there is a positive integer  $c$ such that for every decomposition
$h_q(t)=\tau_1^{s_1}\tau_2^{s_2}\cdots \tau_l^{s_l}$ into a product of powers of distinct
irreducible polynomials, the numbers $s_j$ are bounded by $cq^{n-1}$, for all $q$.

The main task now is to show that for any $g(t)\in k[t]\setminus\{0\}$ we have
$$
(x_1^q,\ldots,x_n^q,f_1,\ldots,f_r):h_q(t)=(x_1^q,\ldots,x_n^q,f_1,\ldots,f_r):h_q(t)g(t).
$$ Let $f\in
(x_1^q,\ldots,x_n^q,f_1,\ldots,f_r):h_q(t)g(t)$. That means
$$
fg(t)h_q(t)\in (x_1^q,\ldots,x_n^q,f_1,\ldots,f_r).
$$
We need to show that $fh_q(t)\in (x_1^q,\ldots,x_n^q,f_1,\ldots,f_r)$.

Since $(x_1^q,\ldots,x_n^q,f_1,\ldots,f_r)$ is homogeneous, we may assume that $f$ is homogeneous
 as well. We may further assume that none of the monomial terms of $f$ is
divisible by any of $x_i^q$ and that the degree $d$ of $f$ is at most
$n(q-1)$. We have a presentation

$fg(t)h_q(t)$
\begin{eqnarray*}
&=&\sum b_ix_i^q+\Big(\sum_{|\mbf{w}|=d-d_1}c_{1,\mbf{w}}\mbf{x}^\mbf{w}\Big)f_1+
\cdots+\Big(\sum_{|\mbf{w}|=d-d_r}c_{r,\mbf{w}}\mbf{x}^\mbf{w}\Big)f_r \qquad (c_{j,\mbf{w}}\in k[t], b_i\in S)\\
&=&\sum b_ix_i^q+\sum_{|\mbf{u}|=d}\Big(\sum_{\mbf{w}+\mbf{v}=\mbf{u}}c_{1,\mbf{w}}A_{1,\mbf{v}}\Big)\mbf{x}^\mbf{u}+
\cdots+\sum_{|\mbf{u}|=d}\Big(\sum_{\mbf{w}+ \mbf{v}=\mbf{u}}c_{r,\mbf{w}}A_{r,\mbf{v}}\Big)\mbf{x}^\mbf{u}\\
&=&\sum b_ix_i^q+\sum_{|\mbf{u}|=d}\Big(\sum_{\mbf{w}+ \mbf{v}=\mbf{u}}c_{1,\mbf{w}}A_{1,\mbf{v}}+
\sum_{\mbf{w}+ \mbf{v}=\mbf{u}}c_{2,\mbf{w}}A_{2,\mbf{v}}+\cdots+\sum_{\mbf{w}+ \mbf{v}=
\mbf{u}}c_{r,\mbf{w}}A_{r,\mbf{v}}\Big)\mbf{x}^\mbf{u}\\
&=&\sum_{|\mbf{u}|=d, ||\mbf{u}||<q}\Big(\sum_{\mbf{w}+ \mbf{v}=\mbf{u}}c_{1,\mbf{w}}A_{1,\mbf{v}}+
\sum_{\mbf{w}+ \mbf{v}=\mbf{u}}c_{2,\mbf{w}}A_{2,\mbf{v}}+ \cdots+\sum_{\mbf{w}+ \mbf{v}=
\mbf{u}}c_{r,\mbf{w}}A_{r,\mbf{v}}\Big)\mbf{x}^\mbf{u}
\end{eqnarray*}
The last equality follows from the assumption that none of the monomial terms of $f$ is
divisible by any of $x_i^q$ for all $i$. It follows from this equality that for each
$\mbf{u}$ with $|\mbf{u}|=d$ and $||\mbf{u}||<q$ we have
$$
\sum_{\mbf{w}+ \mbf{v}=\mbf{u}}c_{1,\mbf{w}}A_{1,\mbf{v}}+\sum_{\mbf{w}+
\mbf{v}=\mbf{u}}c_{2,\mbf{w}}A_{2,\mbf{v}}+\cdots+\sum_{\mbf{w}+ \mbf{v}=\mbf{u}}c_{r,\mbf{w}}A_{r,\mbf{v}}\equiv 0
\qquad \big(\text{mod }g(t)h_q(t)\big),
$$
which implies that
$$
\sum_{\mbf{w}+ \mbf{v}=\mbf{u}}\frac{c_{1,\mbf{w}}}{g(t)}A_{1,\mbf{v}}+\sum_{\mbf{w}+
\mbf{v}=\mbf{u}}\frac{c_{2,\mbf{w}}}{g(t)}A_{2,\mbf{v}}+\cdots+\sum_{\mbf{w}+ \mbf{v}=
\mbf{u}}\frac{c_{r,\mbf{w}}}{g(t)}A_{r,\mbf{v}}
$$
is an element of $k[t]$ and is divisible by $h_q(t)$. Apply Lemma ~\ref{minor}
in this situation, with the matrix $(a_{ij})$ taken as the matrix $M_d$ defined earlier, and the $b_j$ taken
as the \linebreak $\frac{c_{j,\mbf{w}}}{g(t)}$. Note that by
the definition of $h_q(t)$,
the above sum is divisible by each nonzero minor of the matrix $(a_{ij})$. Therefore by Lemma ~\ref{minor} we
can find
$\{c'_{1,\mbf{w}}: |\mbf{w}|=d-d_1\},\ldots,\linebreak\{c'_{r,\mbf{w}}: |\mbf{w}|=d-d_r\}$
in $k[t]$ such that for each $\mbf{u}$ with $|\mbf{u}|=d$ and $||\mbf{u}||<q$,
$$
\sum_{\mbf{w}+ \mbf{v}=\mbf{u}}\frac{c_{1,\mbf{w}}}{g(t)}A_{1,\mbf{v}}+\cdots+
\sum_{\mbf{w}+ \mbf{v}=\mbf{u}}\frac{c_{r,\mbf{w}}}{g(t)}A_{r,\mbf{v}}=
\sum_{\mbf{w}+\mbf{v}=\mbf{u}}c'_{1,\mbf{w}}A_{1,\mbf{v}}+\cdots+\sum_{\mbf{w}+ \mbf{v}=
\mbf{u}}c'_{r,\mbf{w}}A_{r,\mbf{v}}.
$$
This in turn implies that
\begin{eqnarray*}
fh_q(t)&=&\sum_{|\mbf{u}|=d, ||\mbf{u}||<q}\Big(\sum_{\mbf{w}+ \mbf{v}=
\mbf{u}}c'_{1,\mbf{w}}A_{1,\mbf{v}}+\sum_{\mbf{w}+ \mbf{v}=\mbf{u}}c'_{2,\mbf{w}}A_{2,\mbf{v}}+
\cdots+\sum_{\mbf{w}+ \mbf{v}=\mbf{u}}c'_{r,\mbf{w}}A_{r,\mbf{v}}\Big)\mbf{x}^\mbf{u}\\
&=&\sum b'_ix_i^q+\sum_{|\mbf{u}|=d}\Big(\sum_{\mbf{w}+ \mbf{v}=\mbf{u}}c'_{1,\mbf{w}}A_{1,\mbf{v}}+
\sum_{\mbf{w}+ \mbf{v}=\mbf{u}}c'_{2,\mbf{w}}A_{2,\mbf{v}}+\cdots+\sum_{\mbf{w}+\mbf{v}=
\mbf{u}}c'_{r,\mbf{w}}A_{r,\mbf{v}}\Big)\mbf{x}^\mbf{u}\\
&=&\sum b'_ix_i^q+\sum_{|\mbf{u}|=d}\Big(\sum_{\mbf{w}+\mbf{v}=\mbf{u}}c'_{1,\mbf{w}}A_{1,\mbf{v}}\big)\mbf{x}^\mbf{u}+
\cdots+\sum_{|\mbf{u}|=d}\Big(\sum_{\mbf{w}+ \mbf{v}=\mbf{u}}c'_{r,\mbf{w}}A_{r,\mbf{v}}\Big)\mbf{x}^\mbf{u}\\
&=&\sum
b'_ix_i^q+\Big(\sum_{|\mbf{w}|=d-d_1}c'_{1,\mbf{w}}\mbf{x}^\mbf{w}\Big)f_1+\cdots+
\Big(\sum_{|\mbf{w}|=d-d_r}c'_{r,\mbf{w}}\mbf{x}^\mbf{w}\Big)f_r,
\end{eqnarray*}
where the $b_i'$ are polynomials in $S$. Hence $f\in (x_1^q,\ldots,x_n^q,f_1,\ldots,f_r):h_q(t)$.
\end{proof}


The main result of this section is the following

\begin{theorem} Let $k$ be a field of prime characteristic $p$. Fix $n\geq 2$ and set
$$
R=\frac{k[t,x_1,\ldots,x_n]}{(f_1,\ldots,f_r)},
$$
where $t,x_1,\ldots,x_n$ are variables and $f_1,\ldots,f_r$ are homogeneous
polynomials with respect to the grading in which $\deg t=0$ and $\deg x_i=1$ for $i=1,\ldots, n$.
Let $I=(x_1,\ldots,x_n)$. Then there is a constant $c$ such that for each $q=p^e$ there is a
primary decomposition
$$
I^{[q]}=Q_1\cap Q_2\cap\cdots\cap Q_m
$$
with the property that $(\sqrt{Q_i})^{c[q^{n-1}]}\subseteq Q_i$ for all
$i=1,\ldots,m$.
\end{theorem}


\begin{proof} We may assume that the $f_i$ are of positive degrees. Proposition
~\ref{existence} guarantees the existence of polynomials $h_q(t)\in k[t]$ such that for each
$q=p^e$ the ideal $I^{[q]}:h_q(t)$ is primary to the ideal $(x_1,\ldots,x_n)$. Moreover, there
is a constant $c$ with the property that for every irreducible decomposition
$$
h_q(t)=\tau_1^{s_1}\cdots \tau_l^{s_l},
$$
the exponents $s_j$ are bounded above by $cq^{n-1}$. Now the conclusion of the theorem
follows from Lemma ~\ref{stable} and Proposition ~\ref{embedded}.
\end{proof}


\begin{corollary} Let $k$ be a field of prime characteristic $p$ and consider
$$
R=\frac{k[t,x,y]}{(f_1,\ldots,f_r)},
$$
where $f_1,\ldots,f_r$ are homogeneous polynomials with respect to the grading in which $\deg t=0$ and
$\deg x=\deg y=1$. Then the Frobenius powers of the ideal $I=(x,y)$
have linear growth of primary decompositions.
\end{corollary}


\section{Explicit primary decompositions of an example of Singh and Swanson}


Let $k$ be a field of prime characteristic $p$. It was proved in ~\cite{swanson-singh} that,
for the integral domain
$$
R=\frac{k[t,u,v,x,y]}{(u^2x^2+tuxvy+v^2y^2)},
$$
the set
$$
\bigcup_{q=p^e} \ass\text{ }\frac{R}{(x^q,y^q)}
$$
is infinite. By a slight modification of their proofs we show in this section that the same conclusion holds
for the Frobenius powers of the ideal $(u,v,x,y)$, namely that the set
$$
\bigcup_{q=p^e} \ass\text{ }\frac{R}{(u^q,v^q,x^q,y^q)}
$$
is also infinite. Therefore, as mentioned in the introduction of this paper, we want to
analyze primary decompositions of the Frobenius powers of the dimension one ideal \linebreak $I=(u,v,x,y)$
in this hypersurface. We show that they have linear growth property.

To prove the infiniteness of the set of associated primes
mentioned above, we need a modification of the Proposition 2.2 in ~\cite{swanson-singh}.


\begin{proposition}\label{modify}
Let $A$ be an $\mathbb{N}$-graded ring which is generated, as an
$A_0$-algebra, by \linebreak nonzerodivisors $t_1,\ldots,t_n$ of degree 1. Let
$R$ be the extension ring
$$
R=\frac{A[u_1,\ldots,u_n,x_1,\ldots,x_n]}{(u_1x_1-t_1,\ldots,u_nx_n-t_n)}.
$$
Let $m_1,\ldots,m_n$ be positive integers and $f\in A$ a homogeneous
element of degree $r$. Then whenever $k_1,\ldots,k_n$ are positive
integers such that $k_i+m_i > r$ for all $i$, we have
$$
(x_1^{m_1+k_1},\ldots, x_n^{m_n+k_n}, u_1^{m_1+k_1},\ldots, u_n^{m_n+k_n})R:_{A_0} fx_1^{k_1}\cdots
x_n^{k_n}=(t_1^{m_1},\ldots,t_n^{m_n})A:_{A_0} f.
$$
\end{proposition}


\begin{proof}
The inclusion $\supseteq$ is obvious. Define a
$\mathbb{Z}^{n+1}$-grading in $R$ as follows: $\deg x_i=e_i$, \linebreak
$\deg u_i=e_{n+1}-e_i$ and $\deg h=re_{n+1}$ if $h\in A_r$, where $e_i=(0,\ldots, 1,\ldots, 0)$
is the $i$th standard basis element of
$\mathbb{Z}^{n+1}$. Now let $h\in A_0$ be such that
$$hfx_1^{k_1}\cdots x_n^{k_n}\in (x_1^{m_1+k_1},\ldots,x_n^{m_n+k_n}, u_1^{m_1+k_1},\ldots, u_n^{m_n+k_n})R.$$
Then there are homogeneous elements $c_1,\ldots,c_n,b_1,\ldots,b_n$
in $R$ satisfying
$$
hfx_1^{k_1}\cdots x_n^{k_n}=c_1x_1^{m_1+k_1}+\cdots + c_nx_n^{m_n+k_n}+ b_1u_1^{m_1+k_1}
+ \cdots +b_nu_n^{m_n+k_n}.
$$
The degree of the element on the left is
$re_{n+1}+k_1e_1+\cdots+k_ne_n=(k_1,\ldots,k_n,r).$ By the
assumption that $k_i+m_i>r$ for all $i$ we can easily see that
$b_1=\cdots =b_n=0$. The proposition now follows from
\cite[Proposition 2.2]{swanson-singh}.
\end{proof}


We will also make use of the following result from ~\cite{swanson-singh}.

\begin{lemma}\label{ss}\cite[Lemma 4.4(3) and Lemma 3.3]{swanson-singh} Consider the polynomial
ring $k[t,x,y]$. Define the sequence of polynomials $P_0=1, P_1=t, P_{n+1}=tP_n-P_{n-1}$ for $n=1,2,\ldots$.

(a) For any positive integer $n$, one has
$$
(x^n,y^n,x^2+txy+y^2):_{k[t]}xy^{n-1}=(P_{n-1}).
$$

(b) If char $k=p>0$ then the set $\{P_{p^e-2}: e\in\mathbb{N}\}$ has infinitely many irreducible factors.
\end{lemma}


\begin{theorem}\label{infinite}
Let $k$ be a field of prime characteristic $p$ and
$$
R=\frac{k[t,u,v,x,y]}{(u^2x^2+tuxvy+v^2y^2)}.
$$
Then the set
$$
\bigcup_{q=p^e} Ass\text{ }\frac{R}{(u^q,v^q,x^q,y^q)}
$$
is infinite.
\end{theorem}


\begin{proof} For each $q=p^e$ define the ideal
$$
J_q=(u^q,v^q,x^q,y^q)R:_R(ux)(vy)^{q-2}uv.
$$
Set $A_0=k[t]$ and $A=\ds\frac{k[t][a,b]}{a^2+tab+b^2}$. With this setting then the ring R can be presented as
$$
R=\ds\frac{A[u,v,x,y]}{(ux-a,vy-b)}.
$$
Thus applying Proposition \ref{modify} with $f=ab^{q-2}\in A_{q-1}$, we obtain
\begin{eqnarray*}
J_q\cap k[t]&=&(u^q,v^q,x^q,y^q)R:_{k[t]} (ux)(vy)^{q-2}uv\\
&=&(a^{q-1},b^{q-1})A:_{k[t]}ab^{q-2}\\
&=&(P_{q-2}) \qquad\text{by Lemma}~\ref{ss} (a).
\end{eqnarray*}
This implies that
$$(u^q,v^q,x^q,y^q,P_{q-2})\subseteq J_q\subseteq (u,v,x,y,P_{q-2}).
$$
Hence $\Min(J_q)=\Min(u,v,x,y,P_{q-2})$. The set $\bigcup_{e\in\mathbb{N}_+}\Min(u,v,x,y,P_{p^e-2})$ is
 infinite since the set of irreducible factors of the members of $\{P_{p^e-2}: e\in\mathbb{N}_+\}$ is infinite by
Lemma~\ref{ss}(b). Thus the set
$$
\bigcup_{q=p^e}\ass
\frac{R}{(u^q,v^q,x^q,y^q)}
$$
is infinite as well.
\end{proof}


We now try to analyze primary decompositions of the Frobenius powers of the ideal
\linebreak $I=(u,v,x,y)$ in the ring of Theorem~\ref{infinite}. We will actually work in a more general context, namely
in the hypersurfaces
$$
R=\frac{k[t,u,v,x,y]}{(r_0u^2x^2+r_1uxvy+r_2v^2y^2)}
$$
where $r_0,r_1,r_2$ are polynomials in $k[t]$. We will show that the Frobenius powers of the ideal
$I=(u,v,x,y)$ have linear growth of primary decompositions for a large class of polynomials $r_0,r_1,r_2$,
which includes the example above.

For each set of polynomials $r_0,r_1,r_2$ in $k[t]$ we define a generalization of the sequence of
polynomials defined in Lemma~\ref{ss} as follows
\begin{eqnarray*}
P_0&=&1,\\
P_1&=&r_1,\\
P_{n+1}&=&r_1P_n-r_0r_2P_{n-1} \text{ for all } n\geq 1.
\end{eqnarray*}
Note that $P_n$ is just the determinant of the $n\times n$ matrix
$$M_n=\begin{pmatrix}
r_1&r_0&&&\\
r_2&r_1&r_0&&\\
&.&.&.&\\
&&r_2&r_1&r_0\\
&&&r_2&r_1
\end{pmatrix}.
$$

The following lemma generalizes Lemma 4.4(1) from \cite{swanson-singh}.


\begin{lemma}\label{colon}

Consider the polynomial ring $k[t,x,y]$  and an integer $n\geq 1$. Then
$$
 xy^{n-1}P_{n-1}\in (x^n,y^n,r_0x^2+r_1xy+r_2y^2).
$$
\end{lemma}

\begin{proof}
With the grading $\deg t=0, \deg x=\deg y=1$, the module
$$
W=\frac{k[t,x,y]}{(x^n,y^n,r_0x^2+r_1xy+r_2y^2)}
$$
is $\mathbb{N}$-graded. Its $n$th graded component $W_n$, considered as module over $k[t]$, is generated by
$$
x^{n-1}y,x^{n-2}y^2,\ldots,xy^{n-1}.
$$
The relations among these generators come from the equations
$$
x^{n-2-i}y^i(r_0x^2+r_1xy+r_2y^2)=0, \qquad i=0,\ldots, n-2.
$$
This implies that the relation matrix for $W_n$, as a $k[t]$-module, is precisely the
matrix $M_{n-1}$ defined above. Therefore the determinant of this matrix, which is
$P_{n-1}$, annihilates $W_n$, and that implies the conclusion of the lemma.
\end{proof}


We use this lemma to prove the following

\begin{lemma}\label{inclusion} Let $k$ be any field and $S=k[t,u,v,x,y]$. For each integer $n\geq 0$ define
$$
I_n=(u^n,v^n,x^n,y^n,r_0u^2x^2+r_1uxvy+r_2v^2y^2).
$$
Assume furthermore that $2\deg r_1>\deg r_0+\deg r_2$. Then

(a) The polynomial $P_n$ is nonzero for each $n\geq 0$.

(b) Denote by $L_n$ the least common multiple of all (nonzero) polynomials $P_1, P_2,\ldots, P_{n-1}$. We have
$$
(r_0r_2)^{a+b}L_n(ux)^a(vy)^bx^cy^d\in I_n
$$
for all $a,b,c,d\geq 0$ and $2a+2b+c+d=2n$.

(c) $I_n:r_0^nr_2^nL_n\supseteq I_n+(u,v,x,y)^{2n}$.

\end{lemma}


\begin{proof}

(a) It is easy to show, by induction, that the degree of the polynomial $P_n$ is $n\deg r_1$, provided
$2\deg r_1>\deg r_0+\deg r_2$. Hence $P_n$ is a nonzero polynomial for each $n\geq 0$.

(b) We shall prove that for $c\leq d$ we have
$$
r_0^aL_n(ux)^a(vy)^bx^cy^d\in I_n,
$$
and for $c\geq d$,
$$
r_2^bL_n(ux)^a(vy)^bx^cy^d\in I_n.
$$
It is clear that  we may assume that $c\geq d$ without loss of
generality. We shall proceed by induction on $b$. If $b=0$ then $a+c\geq n$,
hence $(ux)^ax^c\in I_n$. The next step in the induction procedure is when $b=1$. We consider two cases,
according as $c>d$ and $c=d$.

If $c>d$ then from $2a+c+d=2n-2$ we get $a+c> n-1$, hence $(ux)^ax^c\in I_n$.
Now suppose $c=d$. Put $m=n-c$. Hence $a+1=m$ and it suffices to show that
$$L_n(ux)^a(vy)\in \big((ux)^m, (vy)^m, r_0(ux)^2+r_1uxvy+r_2(vy)^2\big).$$
But this follows from the Lemma~\ref{colon} above and the definition of $L_n$.

Now let $b\geq 2$ and suppose the conclusion is true for integers smaller than $b$.
Modulo $r_0(ux)^2+r_1uxvy+r_2(vy)^2$, we have
\begin{eqnarray*}
r_2^bL_n(ux)^a(vy)^bx^cy^d&=&-r_2^{b-1}L_n(ux)^a(vy)^{b-2}\big(r_0(ux)^2+r_1uxvy\big)x^cy^d\\
&=& -r_0r_2^{b-1}L_n(ux)^{a+2}(vy)^{b-2}x^cy^d-\\
&&-r_1r_2^{b-1}L_n(ux)^{a+1}(vy)^{b-1}x^cy^d,
\end{eqnarray*}
and by the induction hypothesis this sum belongs to $I_n$.

(c) This follows easily from (b). Indeed, consider $a,b,c,d$ with
$a+b+c+d=2n$, we want to show that $r_0^nr_2^nL_nx^ay^bu^cv^d\in
I_n$. We may assume that $a\geq c $ and $b\geq d$. Then $c+d\leq n$
and
$$r_0^nr_2^nL_nx^ay^bu^cv^d=r_0^nr_2^nL_n(ux)^c(vy)^dx^{a-c}y^{b-d}\in
I_n$$
 by (b).
\end{proof}


\begin{proposition}\label{key} With the same notations as in Lemma ~\ref{inclusion},
for any $g(t)\in k[t]\setminus \{0\}$ and for any positive integer $n$ we have
$$
\big(I_n+(u,v,x,y)^{2n}\big):r_0^{2n}r_2^{2n}g(t)=
\big(I_n+(u,v,x,y)^{2n}\big):r_0^{2n}r_2^{2n}.
$$
\end{proposition}


\begin{proof} We may assume that $r_0r_2$ is nonzero. Let $f(t,u,v,x,y)$ be a homogeneous polynomial
in $u,v,x,y$ of
degree $d<2n$. We need to show that if
$$
f(t,u,v,x,y)g(t)r_0^{2n}r_2^{2n}\in I_n
$$
for some $g(t)\in k[t]\setminus\{0\}$, then
$$
f(t,u,v,x,y)r_0^{2n}r_2^{2n}\in I_n.
$$
We may assume that none of the monomial terms of $f(t,u,v,x,y)$ is divisible by any of $u^n,v^n, x^n,y^n$.
We have a presentation

$f(t,u,v,x,y)g(t)r_0^{2n}r_2^{2n}$

$\quad =Au^n+Bv^n+Cx^n+Dy^n+
\Big(\sum_{i,j,r,s}E^{ir}_{js}u^iv^jx^ry^s\Big)(r_0u^2x^2+r_1uxvy+r_2v^2y^2),
$

\noindent where $A,B,C,D$ are polynomials in $t,u,v,x,y$ which are homogeneous in
$u,v,x,y$, $E_{js}^{ir}$ are polynomials in $k[t]$ and the sum is taken over
all $i,j,r,s\geq 0$ with $i+j+r+s=d-4$. In the sequel when the conditions on the indices are omitted we always
 mean that the sums are taken so that the terms
are of correct degrees. We may write the right-hand side as
\begin{eqnarray*}
&&Au^n+Bv^n+Cx^n+Dy^n+\\
&&+\sum r_0E^{ir}_{js}u^{i+2}v^jx^{r+2}y^s+
\sum r_1E^{ir}_{js}u^{i+1}v^{j+1}x^{r+1}y^{s+1}+\sum r_2E^{ir}_{js}u^iv^{j+2}x^ry^{s+2}\\
&=&Au^n+Bv^n+Cx^n+Dy^n+\\
&&+\sum r_0E^{i-2,r-2}_{js}u^iv^jx^ry^s+
\sum r_1E^{i-1,r-1}_{j-1,s-1}u^iv^jx^ry^s+\sum r_2E^{ir}_{j-2,s-2}u^iv^jx^ry^s.
\end{eqnarray*}
Since we assumed that none of the monomial terms of $f(t,u,v,x,y)$
is divisible by any of $u^n,v^n,x^n,y^n$, the terms that are divisible by any of
$u^n,v^n,x^n,y^n$ in the above sum cancel. We finally get
\begin{equation}
f(t,u,v,x,y)g(t)r_0^{2n}r_2^{2n}=\sum_{\substack{i+j+r+s=d\\ 0\leq i,j,r,s\leq n-1}}\Big(r_0E^{i-2,r-2}_{js}+
r_1E^{i-1,r-1}_{j-1,s-1}+r_2E^{ir}_{j-2,s-2}\Big)u^iv^jx^ry^s,
\end{equation}
in which, by convention, $E^{ir}_{js}=0$ if $\min\{i,j,r,s\}<0$. It
follows from $(1)$ that
\begin{equation}
r_0E^{i-2,r-2}_{js}+r_1E^{i-1,r-1}_{j-1,s-1}+r_2E^{ir}_{j-2,s-2}\equiv 0 \quad \mod
r_0^{2n}r_2^{2n}g(t)
\end{equation}
for all $i+j+r+s=d,\quad 0\leq i,j,r,s\leq n-1$. We claim that in fact $g(t)$
divides each $E^{ir}_{js}$ that appears in the above sums.

We first prove the Proposition assuming the claim. If $g(t)$ divides each $E^{ir}_{js}$, then there are
$\bar{E}^{ir}_{js}$ in $k[t]$ such that $E^{ir}_{js}=g(t)\bar{E}^{ir}_{js}$ for all $i,j,r,s$.
It then follows from (1) that
$$
f(t,u,v,x,y)r_0^{2n}r_2^{2n}=\sum_{\substack{i+j+r+s=d\\ 0\leq i,j,r,s\leq n-1}}\Big(r_0\bar{E}^{i-2,r-2}_{js}+
r_1\bar{E}^{i-1,r-1}_{j-1,s-1}+r_2\bar{E}^{ir}_{j-2,s-2}\Big)u^iv^jx^ry^s.
$$
Now by tracing back the calculations made in the previous paragraphs and using the same regrouping
 arguments as in the end of the proof of Proposition ~\ref{existence} we get
$$
f(t,u,v,x,y)r_0^{2n}r_2^{2n}=\bar{A}u^n+\bar{B}v^n+\bar{C}x^n+\bar{D}y^n+
\Big(\sum\bar{E}^{ir}_{js}u^iv^jx^ry^s\Big)
(r_0u^2x^2+r_1uxvy+r_2v^2y^2),
$$
which of course gives us that
$$
f(t,u,v,x,y)r_0^{2n}r_2^{2n}\in I_n.
$$

We now prove the claim. Define the following sets
\begin{eqnarray*}
A_0&=&\big\{E^{ir}_{js}|\quad i+j+r+s=d-4, \quad  -2\leq i,r < n-2, \quad
0\leq j,s <n\big\},\\
A_1&=&\big\{E^{ir}_{js}|\quad i+j+r+s=d-4, \quad -1\leq i,j,r,s< n-1\big\},\\
A_2&=&\big\{E^{ir}_{js}|\quad i+j+r+s=d-4, \quad 0\leq i,r < n, \quad
-2\leq j,s<n-2\big\}.
\end{eqnarray*}
Note that $A_0,A_1,A_2$ are just the sets of all $E_{js}^{ir}$ that appear as coefficients of $r_0, r_1,r_2$ in (1),
correspondingly. We have
\begin{eqnarray*}
A_1\setminus A_0&=&\big\{E^{ir}_{js}|\quad i+j+r+s=d-4, \quad
n-2\in\{i,r\}\text{ or } -1\in \{j,s\}\big\},\\
A_1\setminus A_2&=&\big\{E^{ir}_{js}|\quad i+j+r+s=d-4, \quad
-1\in\{i,r\}\text{ or } n-2\in \{j,s\}\big\}.
\end{eqnarray*}
Since $i+j+r+s=d-4\leq 2n-5$, $A_1\setminus A_0\cup A_2$ is either
empty or consists of only $E^{ir}_{js}$ with $\min\{i,j,r,s\}<0$ which
are 0. Hence we can conclude that for each nonzero $E^{ir}_{js}$ that appears
in (1), either $E^{ir}_{js}\in A_0$ or $E^{ir}_{js}\in A_2$.

We shall need the following subsets
\begin{eqnarray*}
A_0\setminus A_2&=&\big\{E^{ir}_{js}|\quad i+j+r+s=d-4, \quad -2\leq
i,r<n-2,\quad 0\leq j,s<n,\\
&&\quad
\{-2,-1\}\cap\{i,r\}\neq \emptyset \text{ or } \{n-2,n-1\}\cap \{j,s\}\neq \emptyset\big\},\\
A_2\setminus A_0&=&\big\{E^{ir}_{js}|\quad i+j+r+s=d-4, \quad 0\leq
i,r<n,\quad -2\leq j,s<n-2,\\
&&\quad \{n-2,n-1\}\cap\{i,r\}\neq \emptyset \text{ or }
\{-2,-1\}\cap \{j,s\}\neq \emptyset\big\}.
\end{eqnarray*}

Fix a nonzero $E_{js}^{ir}$, we would like to show that it is divisible by $g(t)$ as claimed. We will consider
three cases, according as $E_{js}^{ir}\in A_0\setminus A_2$, $E^{ir}_{js}\in A_2\setminus A_0$ and $E_{js}^{ir}\in A_0\cap A_2$.

First suppose $E_{js}^{ir}\in A_0\setminus A_2$. Then $\{n-2,n-1\}\cap\{j,s\}\neq\emptyset$. We claim that, for each
integer $l$ such that $E_{j-l,s-l}^{i+l,r+l}$ exists and has nonnegative indices, we have $E_{j-l,s-l}^{i+l,r+l}\in A_0$.
If not, suppose $E^{i+l,r+l}_{j-l,s-l}\in A_2\setminus A_0$ and has nonnegative indices. Then, by the description above,
we have $\{n-2,n-1\}\cap\{i+l,r+l\}\neq\emptyset$. Thus
$$
i+j+r+s=(i+r)+(j+s)\geq (n-2-l)+(n-2+l)=2n-4>d-4,
$$
a contradiction. The relation in (2) can be rewritten as
$$
r_0E^{ir}_{js}\equiv -r_1E^{i+1,r+1}_{j-1,s-1}-r_2E^{i+2,r+2}_{j-2,s-2} \mod  r_0^{2n}r_2^{2n}g(t).
$$
By what we just discussed, the elements $E^{i+1,r+1}_{j-1,s-1}$ and $E^{i+2,r+2}_{j-2,s-2}$ either have some negative indices
or belong to $A_0$. If any of them belong to $A_0$ then by multiplying both sides by $r_0$ we can, using (2), replace
it by combinations of some $E$ with all the subscripts decreased by at least 1. Keep doing this process until we end up
with elements with some negative indices, which are 0. Since each time the subscripts decreased by at least 1, we need
at most $\min\{j,s\}$ steps, hence the power of $r_0$ on the left-hand side will be at most $\min\{j,s\}$+1. We obtain that
$$
r_0^{\min\{j,s\}+1}E_{js}^{ir}\equiv 0 \text{ mod } r_0^{2n}r_2^{2n}g(t).
$$
It follows that $E^{ir}_{js}\equiv 0 \mod g(t)$ since $\min\{j,s\}+1 < 2n$.

By a symmetric argument, if $E^{ir}_{js}\in A_2\setminus A_0$ then
$$
r_2^{\min\{i,r\}+1}E^{ir}_{js}\equiv 0 \mod r_0^{2n}r_2^{2n}g(t),
$$
which results that $E^{ir}_{js}\equiv 0 \mod g(t)$.

Now we consider the last case, namely when $E^{ir}_{js}\in A_0\cap A_2$. We use again
the relation
$$
r_0E^{ir}_{js}\equiv r_1E^{i+1,r+1}_{j-1,s-1}-r_2E^{i+2,r+2}_{j-2,s-2} \mod r_0^{2n}r_2^{2n}g(t).
$$
Each nonzero $E$ on the right-hand side is either in $A_0\cap A_2$ or in $A_0\setminus A_2$
or in $A_2\setminus A_0$. If any of them is in $A_0\cap A_2$, then by multiplying both sides by
$r_0$ we can express it as a combination of other $E$ with subscripts decreased by at least 1.
Iterating this process, we can conclude that there is a relation
$$
r_0^{l}E^{ir}_{js}\equiv \sum_{\tau=1}^{h}c_\tau E^{i+\tau,r+\tau}_{j-\tau,s-\tau}\mod r_0^{2n}r_2^{2n}g(t),
$$
in which $l, h\leq \min\{j,s\}+1$ and each $E^{i+\tau,r+\tau}_{j-\tau,s-\tau}$ is either in
$A_0\setminus A_2$ or in $A_2\setminus A_0$. It follows from the discussion above that for each $\tau$,
$$
r_0^{1+\min\{j-\tau,s-\tau\}}r_2^{1+\min\{i+\tau,r+\tau\}}E^{i+\tau,r+\tau}_{j-\tau,s-\tau}
\equiv 0\mod r_0^{2n}r_2^{2n}g(t).
$$
Hence
$$
r_0^{1+l+\min\{j-1,s-1\}}r_2^{1+\min\{i+h,r+h\}}E^{ir}_{js}\equiv 0\mod r_0^{2n}r_2^{2n}g(t),
$$
and this gives us that $E^{ir}_{js}\equiv 0\mod g(t)$ since
\begin{eqnarray*}
1+\min\{i+h,r+h\}&\leq& 2n,\\
1+l+\min\{j-1,s-1\}&<&2n.
\end{eqnarray*}
This finishes the proof of the claim, hence of the proposition.
\end{proof}


\begin{lemma}\label{primary}
Use the same notation and hypothesis as in Lemma ~\ref{inclusion}(c). We then have
$$
\big(I_n:r_0^{3n}r_2^{3n}L_n\big)=\big(I_n+(u,v,x,y)^{2n}\big):r_0^{2n}r_2^{2n}
$$
 and this ideal is primary to the ideal $(u,v,x,y)$ if $r_0r_2$ is nonzero.
\end{lemma}


\begin{proof} The inclusion
$$
\big(I_n:r_0^{3n}r_2^{3n}L_n\big)\supseteq \big(I_n+(u,v,x,y)^{2n}\big):r_0^{2n}r_2^{2n}
$$
follows from Lemma~\ref{inclusion}. For the converse,
\begin{eqnarray*}
I_n:L_nr_0^{3n}r_2^{3n}&\subseteq&  \Big(\big(I_n+(u,v,x,y)^{2n}\big):r_0^{2n}r_2^{2n}\Big):r_0^{n}r_2^nL_n\\
&=& \big(I_n+(u,v,x,y)^{2n}\big):r_0^{2n}r_2^{2n}\qquad\text{ by
Proposition~\ref{key}}.
\end{eqnarray*}
The conclusion that $\big(I_n+(u,v,x,y)^{2n}\big):r_0^{2n}r_2^{2n}$
is primary to the ideal $(u,v,x,y)$ is another immediate consequence of
Proposition ~\ref{key}.
\end{proof}


We now prove the main theorem of this section.

\begin{theorem}\label{main}
Let $k$ be a field of prime characteristic $p$. Consider the hypersurface
$$
R=\frac{k[t,u,v,x,y]}{(r_0u^2x^2+r_1uxvy+r_2v^2y^2)},
$$
where $r_0,r_1,r_2$ are polynomials in $k[t]$. Let $I=(u,v,x,y)$. Suppose that $r_0r_2$ is nonzero
and that $2\deg r_1>\deg r_0+\deg r_2$. Then there is a constant
$c$ such that for each $q=p^e$, there is a primary decomposition
$$
I^{[q]}=Q_1\cap Q_2\cap\cdots\cap Q_s
$$
satisfying $(\sqrt{Q_i})^{c[q]}\subseteq Q_i$ for all $i=1,\ldots,s$.
\end{theorem}


\begin{proof}  It follows from Lemma ~\ref{primary} that the ideal
$I^{[q]}:r_0^{3q}r_2^{3q}L_q$ is primary to the ideal $(u,v,x,y)$. Recall that the polynomial $L_q$ is defined to be the least common multiple of the polynomials $P_1, P_2,\ldots, P_{q-1}$. It is readily seen that the degrees of the polynomials $P_n$ are bounded above by a linear function in $n$.  Thus there is a constant $c$ such that for each
$q=p^e$ the polynomial $h_q(t)=r_0^{3q}r_2^{3q}L_q$ has an irreducible decomposition
$$
h_q(t)= \tau_1^{s_1}\cdots \tau_l^{s_l}
$$
in which $s_i\leq cq$ for all $i=1,\ldots,l$. Therefore the Frobenius powers of the ideal
$I=(u,v,x,y)$ have linear growth of primary decompositions by Lemma
~\ref{stable} and Proposition ~\ref{embedded}.
\end{proof}


\section{Another example of Singh and Swanson}


Consider the hypersurface
$$
\ds R=\frac{k[t,u,v,w,x,y,z]}{(u^2x^2+v^2y^2+tuxvy+tw^2z^2)}
$$ with $k$ an arbitrary field. This is a UFD and is F-regular if char $k=p>0$. It was proved in \cite[Theorem 6.6]{swanson-singh}
 that if char $k = p>0$ then
the set
$$
\bigcup_{q=p^e} \ass\frac{R}{(x^q,y^q,z^q)}
$$ has infinitely many maximal elements. Again, we can modify their proof to obtain the following


\begin{theorem}
Consider the hypersurface
$$
R=\frac{k[t,u,v,w,x,y,z]}{(u^2x^2+v^2y^2+tuxvy+tw^2z^2)},
$$
where $k$ is a field of prime characteristic $p>0$. Then the set
$$
\bigcup_{q=p^e}Ass\frac{R}{(u^q,v^q,w^q,x^q,y^q,z^q)}
$$
 has infinitely many (maximal) elements.
\end{theorem}


\begin{proof} Consider the ideal
$$
J_q=(u^q,v^q,w^q,x^q,y^q,z^q)R:_R(ux)(vy)^{q-2}uvw^{q-1}.
$$
Set $A_0=k[t]$ and $A=\ds\frac{k[t][a,b,c]}{(a^2+b^2+tab+tc^2)}$. With this setting then the ring $R$ can be represented as
$$
R=\frac{A[u,v,w,x,y,z]}{(ux-a, vy-b, wz-c)}.
$$
Applying Proposition ~\ref{modify} with $f=ab^{q-2}\in
A_{q-1}$, we obtain
\begin{eqnarray*}
J_q\cap k[t]&=&(u^q,v^q,w^q,x^q,y^q,z^q)R:_{k[t]}(ux)(vy)^{q-2}uvw^{q-1}\\
&=&(u^q,v^q,w^q,x^q,y^q,z^q)R:_{k[t]}fuvw^{q-1} \\
&=&(a^{q-1},b^{q-1},c)A:_{k[t]} ab^{q-2}\\
&=&\big(P_{q-2}\big),
\end{eqnarray*}
where $P_{q-2}$ is the polynomial defined in Lemma~\ref{ss}. Now we can now use the same argument
as in the last paragraph of Theorem ~\ref{infinite}.
\end{proof}


However we do not know if the Frobenius powers
of the ideal $(u,v,w,x,y,z)$ have linear growth of primary decomposition. Thus we would like to ask:

\noindent\textbf{Question.} Let $k$ be a field of prime characteristic $p$. Do the Frobenius powers of the ideal
$I=(u,v,w,x,y,z)$ in the hypersurface
$$
\ds R=\frac{k[t,u,v,w,x,y,z]}{(u^2x^2+v^2y^2+tuxvy+tw^2z^2)}
$$
 have linear growth of primary decompositions?


\section{The example of Brenner and Monsky}


Consider the hypersurface
$$
\ds R=\frac{k[t,x,y,z]}{(z^4+z^2xy+zx^3+zy^3+tx^2y^2)},
$$
where $k$ is an algebraically  closed field of characteristic 2. This hypersurface was
first studied by Monsky in~\cite{monsky} and recently it was used again by Brenner-Monsky
to disprove the localization conjecture in tight closure theory. More specifically, Brenner
and Monsky showed that $I_S^*\neq (I_S)^*$ for $I=(x^4,y^4,z^4)$ and $S=k[t]\setminus \{0\}$,
where $(-)^*$ denotes the tight closure operation.

It is reasonable to ask whether there is any example of
ideals that do not have the linear growth property. We restate the theorem of Sharp and Nossem in ~\cite{sharp-nossem}
which relates the linear growth property and the localization
problem.


\begin{theorem}\cite[Theorem 7.14]{sharp-nossem} Let $R$ be a ring of positive characteristic
$p$ and $I$ an ideal of $R$. Suppose that $R$ has a $p^{m_0}$-weak test element, and that the set
$\bigcup_{q=p^e}\ass R/I^{[q]}$ is finite.

Then the tight closure of $I$ commutes with the localization at any
multiplicative subset of $R$ if the Frobenius powers of $I$ have
linear growth of primary decompositions.
\end{theorem}


Consider the ring $R$ and the ideal $I$ in the example of Brenner and Monsky mentioned above.
The ring $R$ is an affine domain over a field, hence it
possesses a test element. Thus, in view of the above theorem, and the result of Brenner
and Monsky, the linear growth property would fail for $I$ if the set
$\bigcup_{q=2^e}\ass R/I^{[q]}$ is finite. Thus it would be interesting to know
the answer to the following question.

\noindent\textbf{Question.} Let $k$ be an algebraically closed field of characteristic 2. Set
$$
R=\frac{k[t,x,y,z]}{(z^4+z^2xy+zx^3+zy^3+tx^2y^2)}.
$$
Is the set
$$
\bigcup_{q=2^e}\ass \frac{R}{(x^q,y^q,z^q)}
$$
finite?


\textbf{Acknowledgement.} The author would like to thank Irena Swanson for suggesting the problem
and for many useful discussions. He would also like to thank Paul C. Roberts for his financial support. The author is grateful to Anurag Singh and the referee for several valuable comments.


\noindent{\bf Remark.} Recently, after the final version of the paper was submitted for publication, the author was able to show, using Monsky's calculations in \cite{brenner-monsky} and \cite{monsky}, that the question above has a negative answer, namely that the set $\bigcup\ass R/(x^q,y^q,z^q)$ is actually infinite.



\noindent \textsc{\small Department of Mathematics, University of
Utah, Salt Lake City, UT 84112}\\
\noindent {\small Email address: \texttt{dinh@math.utah.edu}.}


\end{document}